\pdfoutput=1
\documentclass[11pt,a4paper]{article}
\usepackage{theorem,float}
\usepackage[bookmarks=false]{hyperref} 
\usepackage{graphicx}
\usepackage{amsmath}
\usepackage{amssymb}
\usepackage{mathrsfs}
\usepackage{enumerate}
\usepackage{color}
\usepackage{commath}

\newtheorem{theorem}{\bf Theorem}
\newtheorem{lemma}[theorem]{\bf Lemma}
\newtheorem{remark}[theorem]{Remark}

\newenvironment{proof}{\noindent{\bf Proof}}

\newcommand{\bzero}{\mathbf{0}}
\newcommand{\N}{\mathbb{N}}

\newcommand{\R}{\mathbb{R}}
\newcommand{\cB}{\mathcal{B}}

\newcommand{\bx}{\mbox{\boldmath{$x$}}}

\newcommand{\balpha}{\mbox{\boldmath{$\alpha$}}}

\begin{document}
\title{Approximation of noisy data using
multivariate splines and finite element methods}

\author{Elizabeth Harris \thanks{School of Mathematical \& Physical Sciences,
Mathematics Building,
University of Newcastle,
University Drive,
Callaghan, NSW 2308, Australia, {\tt E.Harris@uon.edu.au}},
\; 
Bishnu P.~Lamichhane \thanks{School of Mathematical \& Physical Sciences,
Mathematics Building,
University of Newcastle,
University Drive,
Callaghan, NSW 2308, Australia, {\tt Bishnu.Lamichhane@newcastle.edu.au}}
\; and 
Quoc Thong Le Gia \thanks{School of Mathematics and Statistics,
University of New South Wales,
Sydney NSW 2052, Australia, {\tt qlegia@unsw.edu.au}}}

\maketitle

\begin{abstract}
We compare a recently proposed multivariate spline based on mixed partial derivatives with two other 
standard splines for the scattered data smoothing problem.
The splines are defined as the minimiser of a penalised least squares functional. The penalties are based on partial differentiation operators, and are integrated using the finite element method. We compare three methods to two problems: to remove the mixture of Gaussian and impulsive noise from an image, and to recover a continuous function from a set of noisy observations.
\end{abstract}

\noindent\textit{Keywords: Sobolev space, scattered data interpolation, Gaussian noise,   impulsive noise}

 \section{Introduction}
We begin by outlining the scattered data problem. Consider the set of scattered points $\mathcal{G} = \left\lbrace \mathbf{p}_i \right\rbrace _{i=1}^N$ in a domain $\Omega \subseteq \R^d$ with $d \in \N$, 
 and the set of noisy observations at those points $\left\lbrace z_i \right\rbrace _{i=1}^N$. We want to reconstruct an unknown function $u$ to approximate the given data. Assuming that the underlying data set is corrupted with Gaussian noise, we can assume that the unknown function $u$ satisfies 
\begin{align*}
z_i = u(\mathbf{p}_i) + n_i, 
\end{align*}
$i = 1, \dots N$, where $\left\lbrace n_i \right\rbrace_{i=1}^N$ is a set of normally distributed random variables with mean $0$ and variance $\sigma^2$.

To recover the unknown function $u$, we will use an approach based on the multivariate L-spline. That is, we will search for a function $u$ that minimises the following least squares functional
\begin{align*}
\sum_{i=1}^N (z_i - u(\mathbf{p}_i))^2 + \lambda \int_\Omega (Lu(x,y))^2 \mathop{d\mathbf{x}}
\end{align*}
over a Sobolev space $V$, where $L$ is a partial differentiation operator, and $\lambda$ is a positive smoothing parameter. 

We use  the standard notation for Sobolev spaces on $\Omega$ \cite{Ada75,BS94,Cia78}. 
We consider three different choices for $L$. The first choice is to take $Lu$ as the gradient of $u$. Then we need to have $V = H^1(\Omega)$. However, the continuous problem is not well-posed with this choice for $d\geq 2$, because the point value of a function is not defined in $H^1(\Omega)$ when $d \geq 2$. The second choice is to choose $Lu$ as the Laplacian of $u$. Then we need to have $V = H^2(\Omega)$. Again, the continuous problem is not well-posed for $d>3$, because the point value of a function is not defined in $H^2(\Omega)$ when $d > 3$. The third choice is to include mixed partial derivatives of $u$ 
on the gradient penalty to construct $Lu$ \cite{LRH14}. Unlike the other two choices, the resulting spline is well defined for any dimension $d \in \N$. 

This is the first time that computational results of the newly proposed multivariate spline \cite{LRH14}  are 
presented and compared with other existing techniques. Moreover, we observe the instability of the 
gradient penalty approach in our numerical experiments, which is another novelty of this contribution.

We will apply these methods to two problems. The first problem is to recover an image that have been corrupted with both Gaussian and impulsive noise. We apply a finite element method to compute the solution of the above minimisation problem.  Finite element methods have recently become popular in 
different areas of image processing \cite{PR00,FNM01,Bes04,WQZ06,DI06}. 
Finite element methods are applied in 
\cite{Lam09} and \cite{Lam14} to remove the mixture of Gaussian and impulsive noise using the gradient penalty 
and total variation penalty, respectively. 

The second problem is to recover a continuous function from a set of noisy observations. We consider observations that have been corrupted with Gaussian noise. In this example, we see spurious spikes in the solution using 
the gradient penalty. This is due to the fact that the gradient penalty does not control the point-wise values of
the function. Numerical results show that we can increase the mesh-size to reduce the height of the spikes but they cannot be  totally removed. 

This paper is organised as follows. In the next section, we present  the gradient penalty smoothing technique. In the third section, we present the smoothing technique based on the minimisation of a functional involving mixed partial derivatives. 
In the fourth section, we compare the three finite element methods in denoising images and recovering continuous functions. We discuss these results in the last section.

\section{Multivariate Spline with Gradient Penalty}


The  multivariate spline with the gradient penalty is the following minimisation problem 
 \begin{align}
\min_{u \in V} \left( \sum_{i=1}^N (u(\mathbf{p}_i)-z_i)^2 + \lambda\int_{\Omega} \norm{\nabla u}^2 \mathop{dx}\mathop{dy} \right).
\end{align}
 Due to the choice of the minimisation functional it is natural to take $V = H^1(\Omega)$, for which the 
problem will not be well-posed when $d>1$.  

Now we consider a finite element
discretisation of the spline.  Let $C^0(\Omega)$ be the space of continuous functions in $\Omega$ and 
 ${\cal T}_h$ a finite element  triangulation of $\Omega$.
Note that ${\cal T}_h$ is the set of triangles or rectangles. 
Then let
\begin{equation}\label{fespace}
V_h=\{u_h\in C^0(\Omega)|\, u_h{|_{T}} \in P(T),\, T \in {\cal T}\}
\end{equation}
be a finite element space, where $P(T)$ is the linear polynomial space 
if $T$ is a triangle, and $P(T)$ is the bilinear polynomial  
space on $T$ if $T$ is a rectangle \cite{QV94}. 
The minimisation problem leads to the variational problem of finding 
$u_h \in V_h$ such that 
\[ a(u_h,v_h) = \ell(v_h),\quad v_h \in V_h,\]
where the bilinear form $a(\cdot,\cdot)$ and the linear form $\ell(\cdot)$ are given by
\begin{align*}
a(u,v) &= \sum_{i=1}^N u(\mathbf{p}_i) v(\mathbf{p}_i) + \lambda \int_{\Omega} \nabla u \cdot \nabla v \mathop{d\mathbf{x}},\\
\ell(v) &= \sum_{i=1}^N v(\mathbf{p}_i)z_i.
\end{align*}
It is easy to show that the above problem has a unique solution under the assumption that 
 the set of scattered points $\mathcal{G}$ is non-empty \cite{Lam09}.

%

Since $a(\cdot,\cdot)$ is positive definite, we can define the energy $\norm{\cdot}_a$ on $V_h$ as 
\begin{align*}
\norm{v_h}_a^2 = a(v_h,v_h)
\end{align*}
for all $v \in V$.

The following lemma shows that the discrete multivariate spline with the gradient penalty is well-posed 
for $d=1$ but not well-posed for $d>1$. The point-value of a function is not controlled by 
the gradient of the function when $d>1$.  The 
well-posedness is exhibited in the stability result first proved in \cite{GH09}. 
For completeness we have given these results in Lemmas 1, 2 and 3, which are taken from 
\cite{GH09}.

\begin{lemma}\label{lem2}
(Discrete Sobolev inequality). There exists constant $c_d > 0$ such that for all $u \in V_h$
\begin{enumerate}
\item $\abs{u(x)} \leq c_d \norm{u}_{H^1(\Omega)}$ for $d=1$.
\item $\abs{u(x)} \leq c_d \left( 1 + \abs{\log h} \right) \norm{u}_{H^1(\Omega)}$ for $d = 2$.
\item $\abs{u(x)} \leq c_d h^{1-d/2} \norm{u}_{H^1(\Omega)}$ for $d > 2$.
\end{enumerate}
The constant $c_d$ is independent of the mesh-size $h$ but depends on $d$.
These bounds are tight, and for $d > 2$ we have that
\begin{align*}
c_d \geq \frac{1}{\sqrt{3d + 1}}\left( \frac{3}{2} \right)^{d/2}.
\end{align*}
\end{lemma}

\begin{lemma}\label{lem3} (Discrete Poincar\'{e} inequality). Let $(x_0,y_0) \in \Omega$ and $u_0 = u(x_0,y_0)$ for $u \in V_h$. Then there exist constants $c_d > 0$ such that
\begin{enumerate}
\item $\norm{u-u_0}_{L^2(\Omega)} \leq c_d \norm{\nabla u}_{L^2(\Omega)}$ for $d = 1$
\item $\norm{u-u_0}_{L^2(\Omega)} \leq c_d \left( 1 + \abs{\log h}\right) \norm{\nabla u}_{L^2(\Omega)}$ for $d = 2$
\item $\norm{u-u_0}_{L^2(\Omega)} \leq c_d h^{1-d/2} \norm{\nabla u}_{L^2(\Omega)}$ for $d > 2$.
\end{enumerate}
\end{lemma}
\begin{lemma}\label{lem4} (Discrete V-ellipticity). There exist constants $c_d$ and $C_d$ such that the energy norm on $V_h$ satisfies
\begin{align*}
\alpha_{d,h} \norm{u}_{H^1(\Omega)} \leq \norm{u}_a \leq \beta_{d,h} \norm{u}_{H^1(\Omega)}
\end{align*} 
for all $u \in V_h$, where $\alpha_{d,h}$ and $\beta_{d,h}$ are given by
\begin{enumerate}
\item $\alpha_{d,h} = \left( c_d \lambda^{-1/2} + \lambda^{-1/2} + 1 \right)^{-1}$ and $\beta_{d,h} = C_d + \sqrt{\lambda}$ for $d = 1$
\item $\alpha_{d,h} = \left( c_d \lambda^{-1/2} \left( 1 + \abs{\log h} \right) + \lambda^{-1/2} + 1 \right)^{-1}$ and $\beta_{d,h} = C_d \left( 1 + \abs{\log h} \right) + \sqrt{\lambda}$ for $d = 2$
\item $\alpha_{d,h} = \left( c_d \lambda^{-1/2} h^{1-d/2} + \lambda^{-1/2} + 1 \right)^{-1}$ and $\beta_{d,h} = C_d h^{1-d/2} + \sqrt{\lambda}$ for $d > 2$.
\end{enumerate}
\end{lemma}
The above results imply that for the solution $u_h \in V_h$ of 
the spline with the gradient penalty we have 
\begin{align*}
\norm{u_h}_a \leq \frac{\beta_{d,h}}{\alpha_{d,h}} \norm{\ell}_{L^2(\Omega)}.
\end{align*}
\begin{remark}
We can see that the ill-posedness is exhibited in the stability constant being not 
independent of the mesh-size $h$. There is no easy way to remove this dependency.
\end{remark}

\section{New Multivariate Spline with Mixed Derivative Penalty}
In order to define the new multivariate spline, we define the associated 
Sobolev space. 
Let $\cB = \{0,1\}^d\backslash \{\bzero\}$, where 
$\bzero \in \R^d$ is a zero vector.
We use a standard multi-index notation with 
$\balpha = (\alpha_1,\cdots,\alpha_d)\in \cB$ so that 
a mixed derivative of  a sufficiently smooth function $u$ is denoted by 
\[ D^{\balpha} u =  \frac{\partial^{\sum_{i=1}^d\alpha_i} u}{\partial x_1^{\alpha_1}\cdots 
\partial x_d^{\alpha_d}},\]
where  we use the usual Cartesian coordinate system with 
$\bx =(x_1,\cdots,x_d) \in \R^d$.

We now define our Sobolev space for the multivariate spline problem as 
 \[  H_m^1(\Omega) := \left\{ u \in L^2(\Omega):\, D^{\balpha}u \in L^2(\Omega),\,
 \balpha \in \cB\right\},\]
which is equipped with the norm 
 \[ \|u\|_{H_m^1(\Omega)} = \sqrt{  \|u\|^2_{L^2(\Omega)}  + 
  \sum_{\balpha \in \cB}\|D^{\balpha} u\|^2_{L^2(\Omega)}}.
  \]
We note that the space 
$H^1_m(\Omega)$ is a Hilbert space,  and $H^1_m(\Omega) \subset C^0(\Omega)$
\cite{ST87}. The new multivariate spline is then obtained as a solution of 
the minimisation problem 
\begin{align}
\min_{u \in H^1_m(\Omega)} \left( \sum_{i=1}^N (u(\mathbf{p}_i)-z_i)^2 + \lambda\|u\|_{H^1_m(\Omega)}^2 \right).
\end{align}

For $d=2$ we write our bivariate spline as
\begin{align}
\min_{u \in H^1_m(\Omega)} \left( \sum_{i=1}^N (u(\mathbf{p}_i)-z_i)^2 + \lambda\int_{\Omega} \norm{\nabla u}^2 + \left(\frac{\partial^2 u}{\partial x \partial y}\right)^2 \mathop{d\mathbf{x}} \right).
\end{align}

 We now introduce a bilinear form $b(\cdot,\cdot)$ and a linear form $\ell(\cdot)$, given by
\begin{align*}
b(u,v) &= (Pu)^T(Pv) + \lambda\int_{\Omega} \nabla u^T \nabla v + \frac{\partial^2 u}{\partial x \partial y} \frac{\partial^2 v}{\partial x \partial y} \mathop{d\mathbf{x}}, \\
\ell(v) &= (Pv)^T \mathbf{z},
\end{align*}
where
\begin{align*}
Pu = (u(x_1,y_1),u(x_2,y_2),\dots, u(x_N,y_N))^T
\end{align*}
is a column vector of the function values of $u$ at the scattered points $\mathcal{G}=\left\lbrace \mathbf{p}_i \right\rbrace_{i=1}^N$, and $\mathbf{z} \in \R^N$ is a column vector with $i$th component $z_i$. Then the continuous problem is to find $u \in V$ such that
\begin{align}
b(u,v) = \ell(v)
\label{b(u,v)}
\end{align}
for all $v \in V$.

Let $\Omega$ be a rectangle in $\R^2$. Then let $\mathcal{T}_h$ be the tensor product partition of the domain with mesh size $h$, such that each element $T \in \mathcal{T}_h$ is a rectangle. Then we define a finite element space $V_h$ as
\begin{align*}
V_h = \left\lbrace u_h \in C^0(\Omega) : u_h|_T \in \mathcal{P}(T), T \in \mathcal{T}_h \right\rbrace,
\end{align*}

where $\mathcal{P}(T)$ is the space of bilinear polynomials on $T$. We can now write our discrete problem as
\begin{align*}
\min_{u_h \in V_h} \left( \sum_{i=1}^N (u_h(\mathbf{p}_i)-z_i)^2 + \lambda\int_{\Omega} \norm{\nabla u_h}^2 + \left(\frac{\partial^2 u_h}{\partial x \partial y}\right)^2 \mathop{d\mathbf{x}} \right).
\end{align*}
That is, the discrete problem is to find $u_h \in V_h$ such that
\begin{align}
b(u_h,v_h) = \ell(v_h)
\label{b(uh,vh)}
\end{align}
for all $v_h \in V_h$.

The discrete problem is shown to be well-posed in \cite{Lam14}. Here 
we recall some of the important results. 
We first show that the bilinear form $b(\cdot,\cdot)$ is positive definite on $V_h$.
\begin{lemma}
Let $\lambda > 0$ and let the set of scattered points $\mathcal{G}$ be non-empty. Then the bilinear form $b(\cdot,\cdot)$ is positive definite on the vector space $V_h$.
\end{lemma}
\begin{proof}
If $u_h = 0$, then clearly $b(u_h,u_h) = 0$. Conversely, let $b(u_h,u_h) = 0$. Then
\begin{align*}
Pu_h = 0, \text{ } \nabla u_h = 0, \text{ and } \frac{\partial^2 u_h}{\partial x \partial y} = 0.
\end{align*}

Since $u_h$ is a continuous function, $\nabla u_h = 0$ gives that $u$ is a constant function in $\Omega$. Further, since $\mathcal{G}$ is non-empty and $Pu_h = 0$, we have that $u_h=0$.
\end{proof}

Since $b(\cdot,\cdot)$ is positive definite, we can define the energy norm $\norm{\cdot}_b$ on $V_h$ as
\begin{align*}
\norm{v_h}_b^2 = b(v_h,v_h)
\end{align*}
for all $v_h \in V_h$. Since $b(\cdot,\cdot)$ and $\ell(\cdot)$ satisfy the conditions of the Lax-Milgram lemma \cite{BS94},\cite{Cia78}, the unique minimiser is the solution of the discrete problem \eqref{b(uh,vh)}.  In addition, the following holds.

\begin{lemma}
Let $\lambda > 0$ and let $\mathcal{G}$ be non-empty. Then the discrete problem \eqref{b(uh,vh)} admits a unique solution which depends continuously on the data with respect to the energy norm $\norm{\cdot}_b$.
\end{lemma}
\begin{proof}
We have that
\begin{align*}
\abs{b(u_h,v_h)} &\leq \norm{u_h}_b \norm{v_h}_b, \text{ and}\\
\abs{\ell(v_h)} &\leq \norm{\mathbf{z}} \norm{v_h}_b
\end{align*}
for all $u_h,v_h \in V_h$. Hence $b(\cdot,\cdot)$ and $\ell(\cdot)$ are continuous on $V_h$. We also have that
\begin{align*}
b(u_h,u_h) = \norm{u_h}_b^2
\end{align*}
for all $u_h \in V_h$. Hence $b(\cdot,\cdot)$ is coercive on $V_h$. By the Lax-Milgram lemma \cite{BS94},\cite{Cia78}, there exists a unique solution $u_h$ of the discrete problem \ref{b(uh,vh)}. Additionally, the solution depends continuously on the data $\mathbf{z}$.
\end{proof}

In addition, a direct application of the C\'{e}a lemma provides an optimal \textit{a priori} estimate of the discrete solution.
\begin{lemma}
Let $u$ be the solution to the continuous problem \eqref{b(u,v)}, and let $u_h$ be the solution to the discrete problem \eqref{b(uh,vh)}. Then
\begin{align*}
\norm{u-u_h}_b \leq \inf_{v_h \in V_h} \norm{u-v_h}_b.
\end{align*}
\end{lemma}


Each finite element basis function is associated with a point in the tensor product partition $\mathcal{T}_h$. Assuming there are $mn$ points, we have $mn$ basis functions. Let $\left\lbrace \phi_i \right\rbrace_{i=1}^{mn}$ be the set of finite element basis functions, which span $V_h$. Then we can write our solution $u_h \in V_h$ as a linear combination of these basis functions, namely
\begin{align*}
u_h(x,y) = \sum_{i=1}^{mn} u_i\phi_i(x,y).
\end{align*}

Let $\mathbf{u} = (u_1,u_2,\dots,u_{mn})^T$ and let $K$ be the finite element stiffness matrix, where $K_{ij} = \int_\Omega \nabla u_h^T \nabla v_h \mathop{d\mathbf{x}}$. Let $M$ be a mixed partial derivative matrix, where $M_{ij} = \int_\Omega \frac{\partial u_h}{\partial x \partial y} \frac{\partial v_h}{\partial x \partial y} \mathop{d\mathbf{x}}$. Then we want to find the solution to the linear system
\begin{align*}
(A^T A + \lambda(K + M))\mathbf{u} = A^T \mathbf{z},
\end{align*}

where $A$ is a matrix of size $N \times mn$, with entries $A_{ij} = \phi_j(\mathbf{p}_i)$.
\section{Numerical Results}
\subsection{Real life images}

We would like to recover some real life images. Consider an image of size $m \times n$. Then we define a tensor product partition $\mathcal{T}_h$ of the square $[0,1] \times [0,1]$ using the collection of points 
\begin{align*}
\mathcal{N}_h= \left\lbrace (a_i,b_j) \right\rbrace_{i=1,j=1}^{n,m}, \quad
\text{where } a_i = \frac{i-1}{n-1}, \;
\text{and } b_j = \frac{j-1}{m-1}.
\end{align*}
Then each pixel of the image is associated with a grid point in $\mathcal{N}_h$.

Since we know the images before they have noise applied to them, we will use peak signal-to-noise ratio (PSNR) to compare the results. Let the original image be given by $I$, and the recovered image be given by $\hat{I}$. Then the PSNR is given by
\begin{align*}
\text{PSNR} 
= 10\cdot \log_{10}\left( \frac{\text{MAX}_{I}^2}{\text{MSE}} \right) 
= 20\cdot \log_{10}\left( \frac{\text{MAX}_{I}}{\sqrt{\text{MSE}}} \right),
\end{align*}

where $\text{MAX}_{I}$ is the maximum pixel value of the image, and MSE is the mean square error. We note that MSE is given by
\begin{align*}
\text{MSE}
= \frac{1}{mn} \sum_{i=1}^{m} \sum_{j=1}^{n} \norm{I_{ij}-\hat{I}_{ij}}^2.
\end{align*}

We now consider two test images. These images are the Lena image and Baboon image (see Figure \ref{fig:10}). We will apply both Gaussian and impulsive noise to these images. The Gaussian noise has zero mean and variances 0.05 and 0.1, and the salt and pepper noise has densities from $30\%$ through to $80\%$. 

We will now use the three different splines to reconstruct the images. As an example, we will first consider the images corrupted with Gaussian noise with variance 0.05, and impulsive noise of density 60\%. In the first image of Figure \ref{fig:11}, we show the noisy Lena image. The next three images show the reconstructed images obtained by the three splines. The results for the Baboon image are shown in Figure \ref{fig:12}.

\begin{figure}[H]
\begin{center}
\includegraphics[scale=0.2]{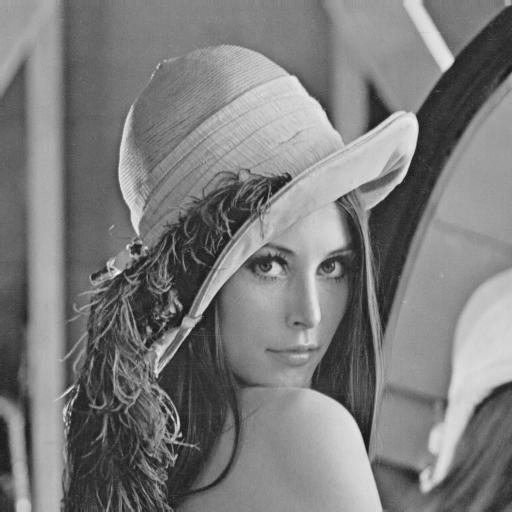}
\hspace{5mm}
\includegraphics[scale=0.2]{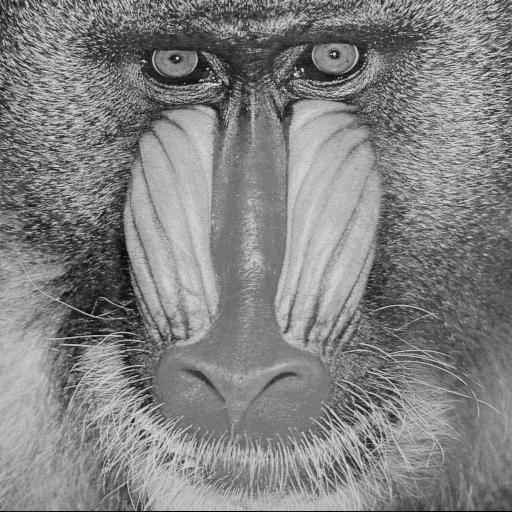}
\caption{Lena image (left), Baboon image (right)}
\label{fig:10}
\end{center}
\end{figure}

\begin{figure}[H]
\begin{center}
\includegraphics[scale=0.54,trim={14cm 10cm 4.6cm 9.2cm},clip]{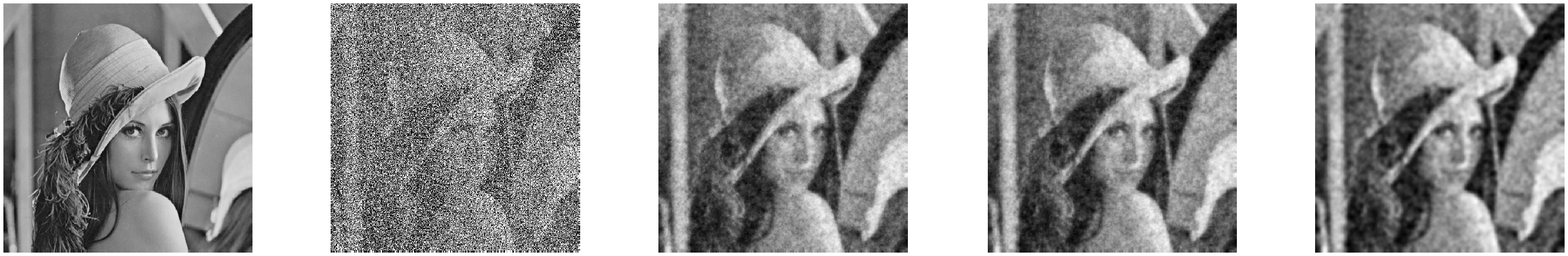}
\caption{Noisy Lena image ($\sigma^2 = 0.05$, $\text{d}=60$) (first), recovered image using gradient penalty spline (second), recovered image using mixed derivative spline (third), recovered image using biharmonic spline (fourth)}
\label{fig:11}
\end{center}
\end{figure}

\begin{figure}[H]
\begin{center}
\includegraphics[scale=0.54,trim={14cm 10cm 4.6cm 9.2cm},clip]{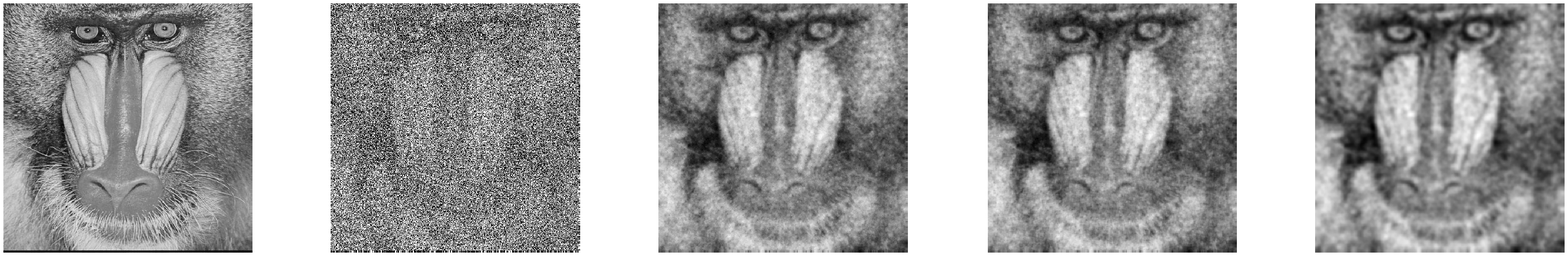}
\caption{Noisy Baboon image ($\sigma^2 = 0.05$, $\text{d}=60$) (first), recovered image using gradient penalty spline (second), recovered image using mixed derivative spline (third), recovered image using biharmonic spline (fourth)}
\label{fig:12}
\end{center}
\end{figure}

We will now show the PSNR for the reconstructed images in Tables 1-4 below.

\begin{table}[H]
\begin{center}
\caption{Lena image PSNR for Gaussian (variance 0.05) and impulsive noises.}
\begin{tabular}{ |c|c|c|c|c|c|c| }
\hline
& \multicolumn{6}{|c|}{Lena image PSNR} \\
\hline
& \multicolumn{6}{|c|}{Noise density} \\
\hline
& 30\% & 40\% & 50\% & 60\% & 70\% & 80\% \\
\hline
Grad. & 22.13 & 22.01 & 22.22 & 21.90 & 21.42 & 20.91 \\ 
Mixed & 22.29 & 22.28 & 22.42 & 21.93 & 21.59 & 20.80 \\ 
Biharm. & 22.98 & 22.95 & 22.30 & 22.09 & 21.74 & 21.48 \\ 
 \hline
\end{tabular}
\label{tab:1}
\end{center}
\end{table}

\begin{table}[H]
\begin{center}
\caption{Baboon image PSNR for Gaussian (variance 0.05) and impulsive noises.}
\begin{tabular}{ |c|c|c|c|c|c|c| }
\hline
& \multicolumn{6}{|c|}{Baboon image PSNR} \\
\hline
& \multicolumn{6}{|c|}{Noise density} \\
\hline
& 30\% & 40\% & 50\% & 60\% & 70\% & 80\% \\
\hline
Grad. & 19.14 & 19.09 & 18.81 & 18.48 & 18.47 & 18.37 \\ 
Mixed & 19.14 & 19.08 & 18.74 & 18.48 & 18.35 & 18.32 \\ 
Biharm. & 18.89 & 18.71 & 18.47 & 18.12 & 17.84 & 17.83 \\ 
 \hline
\end{tabular}
\label{tab:2}
\end{center}
\end{table}

\begin{table}[H]
\begin{center}
\caption{Lena image PSNR for Gaussian (variance 0.1) and impulsive noises.}
\begin{tabular}{ |c|c|c|c|c|c|c| }
\hline
& \multicolumn{6}{|c|}{Lena image PSNR} \\
\hline
& \multicolumn{6}{|c|}{Noise density} \\
\hline
& 30\% & 40\% & 50\% & 60\% & 70\% & 80\% \\
\hline
Grad. & 21.45 & 21.40 & 20.81 & 20.34 & 20.21 & 19.82 \\ 
Mixed & 21.77 & 21.69 & 20.80 & 20.44 & 20.21 & 19.85 \\ 
Biharm. & 22.20 & 22.09 & 21.56 & 21.10 & 20.48 & 20.14 \\ 
 \hline
\end{tabular}
\label{tab:3}
\end{center}
\end{table}

\begin{table}[H]
\begin{center}
\caption{Baboon image PSNR for Gaussian (variance 0.1) and impulsive noises.}
\begin{tabular}{ |c|c|c|c|c|c|c| }
\hline
& \multicolumn{6}{|c|}{Baboon image PSNR} \\
\hline
& \multicolumn{6}{|c|}{Noise density} \\
\hline
& 30\% & 40\% & 50\% & 60\% & 70\% & 80\% \\
\hline
Grad. & 18.24 & 18.20 & 17.93 & 18.09 & 17.89 & 17.65 \\ 
Mixed & 18.18 & 18.17 & 17.89 & 18.06 & 17.86 & 17.56 \\ 
Biharm. & 18.06 & 17.85 & 17.66 & 17.54 & 17.40 & 17.06 \\ 
 \hline
\end{tabular}
\label{tab:4}
\end{center}
\end{table}

Note that we have chosen our parameter $\lambda$ using generalised cross validation \cite{Wah90} and the stochastic trace estimator proposed by Hutchinson \cite{Hut89}. We note that this gives a good estimate of the optimal parameter. In Figure \ref{fig:PSNR GCV} we have plotted the PSNR and the generalised cross validation function versus $\lambda$. Note that the validation function has been scaled for visualisation purposes. For both plots, the Lena image has been corrupted with Gaussian noise with variance 0.05, and has been recovered with the mixed derivative spline. In the left plot, the image has been corrupted with impulsive noise of density 30\% while in the right plot the density is 40\%.

\begin{figure}[H]
\begin{center}
\includegraphics[scale=0.32,trim={0.5cm 1.5cm 1.5cm 1.6cm},clip]{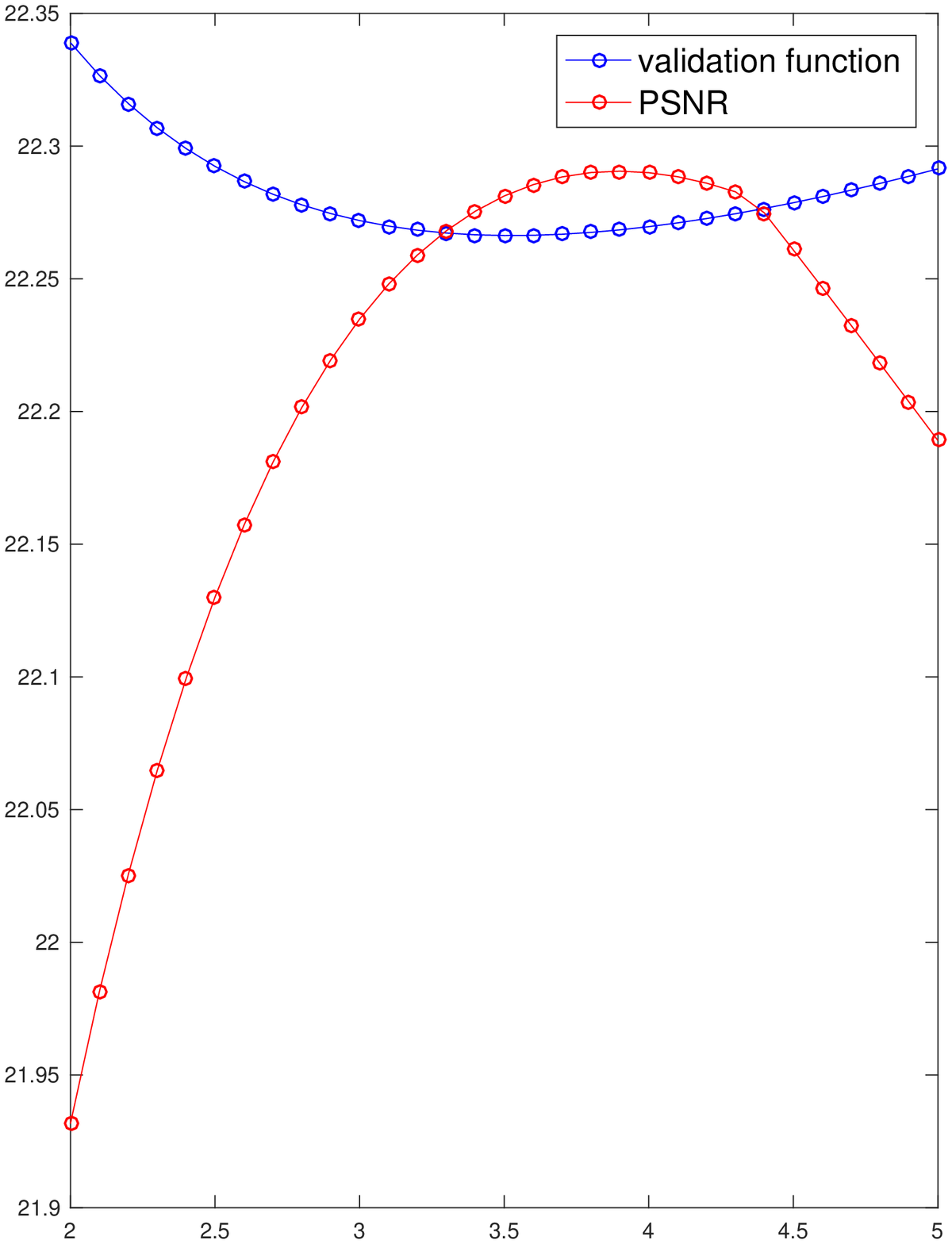}
\hspace{5mm}
\includegraphics[scale=0.32,trim={0.5cm 1.5cm 1.5cm 1.6cm},clip]{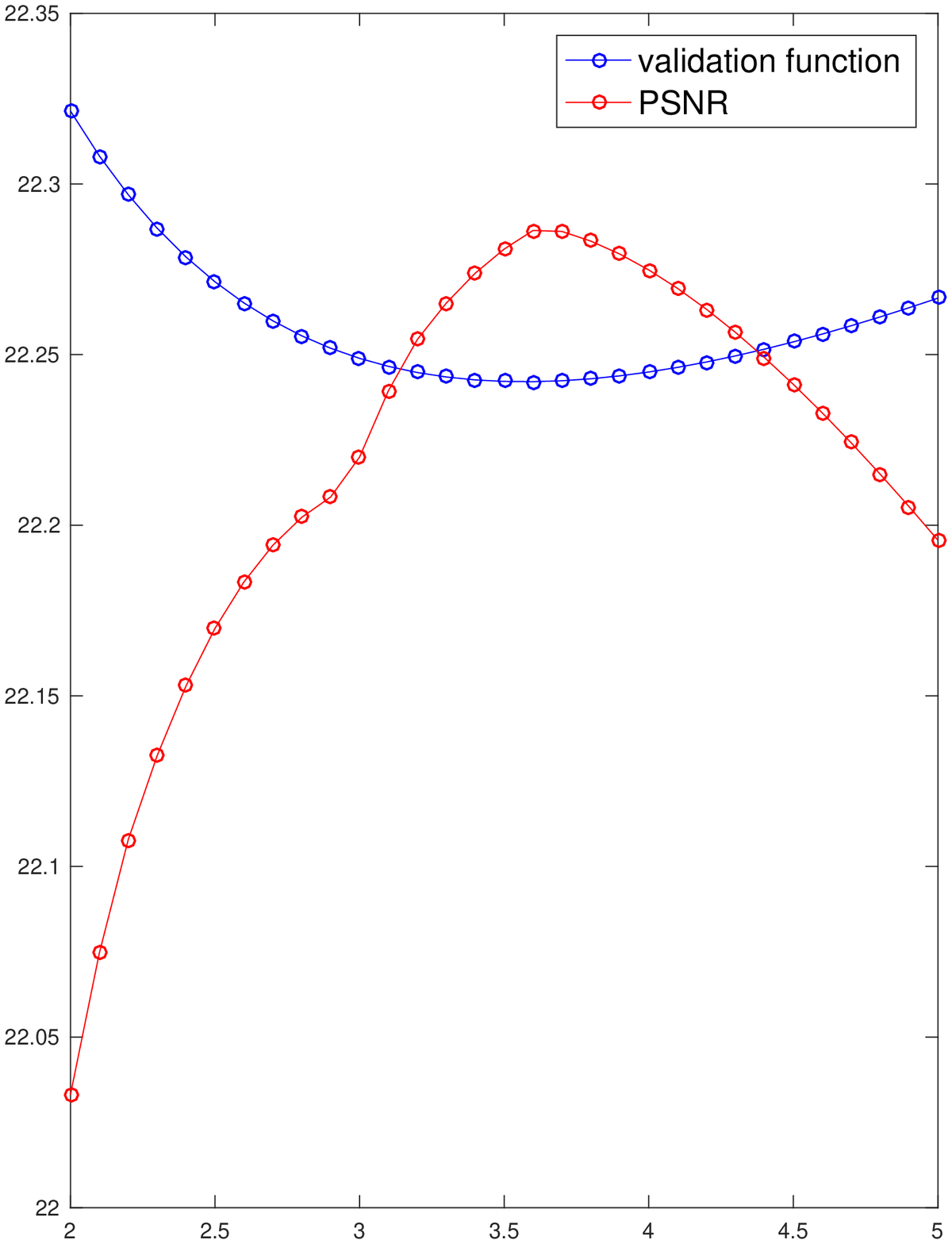}
\caption{Generalised cross validation function and PSNR versus $\lambda$ for Gaussian noise with variance 0.05 and impulsive noise with densities 30\% (left) and 40\% (right).}
\label{fig:PSNR GCV}
\end{center}
\end{figure}

\subsection{Binary image}

We will now apply the same methods to a binary test image (see Figure \ref{fig:13}). As an example, consider the image is corrupted with Gaussian noise of variance $0.05$, and impulsive noise of density $60\%$. We show the noisy image and the reconstructed images in Figure \ref{fig:14}.

\begin{figure}[H]
\begin{center}
\includegraphics[scale=0.32]{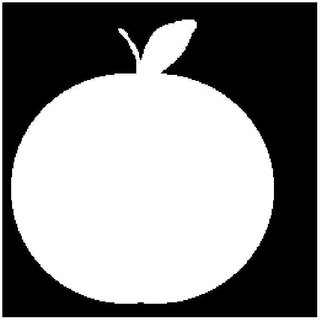}
\caption{Binary image}
\label{fig:13}
\end{center}
\end{figure}

\begin{figure}[H]
\begin{center}
\includegraphics[scale=0.43,trim={6.3cm 9cm 0cm 8.7cm},clip]{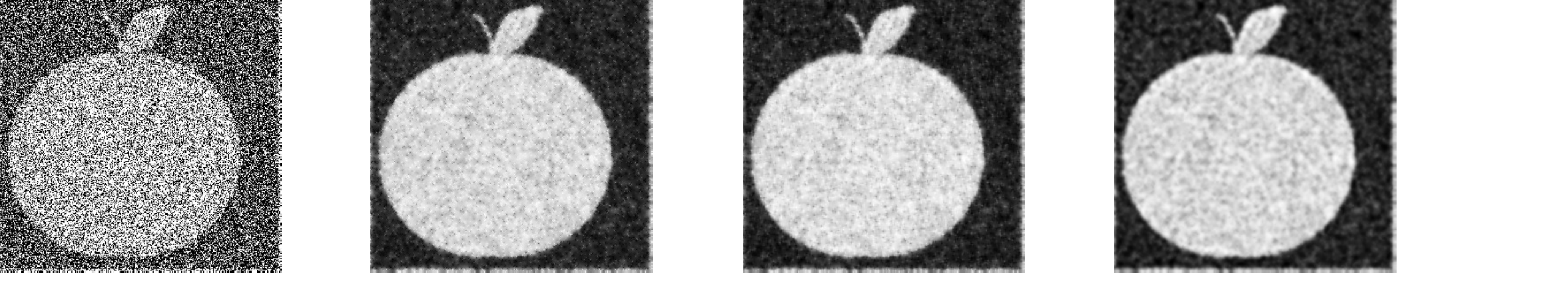}
\caption{Noisy Binary image ($\sigma^2 = 0.05$, $\text{d}=60$) (first), recovered image using gradient penalty spline (second), recovered image using mixed derivative spline (third), recovered image using biharmonic spline (fourth)}
\label{fig:14}
\end{center}
\end{figure}

We will now show the PSNR for the reconstructed image in Tables 5 and 6 below.

\begin{table}[H]
\begin{center}
\caption{Binary image PSNR for Gaussian (variance 0.05) and impulsive noises.}
\begin{tabular}{ |c|c|c|c|c|c|c| }
\hline
& \multicolumn{6}{|c|}{Binary image PSNR} \\
\hline
& \multicolumn{6}{|c|}{Noise density} \\
\hline
& 30\% & 40\% & 50\% & 60\% & 70\% & 80\% \\
\hline
Grad. & 14.77 & 14.66 & 14.37 & 14.19 & 13.77 & 13.47 \\ 
Mixed & 15.37 & 15.21 & 15.03 & 14.84 & 14.43 & 14.00 \\ 
Biharm. & 15.17 & 14.93 & 14.86 & 14.49 & 14.39 & 13.72 \\ 
 \hline
\end{tabular}
\label{tab:5}
\end{center}
\end{table}

\begin{table}[H]
\begin{center}
\caption{Binary image PSNR for Gaussian (variance 0.1) and impulsive noises.}
\begin{tabular}{ |c|c|c|c|c|c|c| }
\hline
& \multicolumn{6}{|c|}{Binary image PSNR} \\
\hline
& \multicolumn{6}{|c|}{Noise density} \\
\hline
& 30\% & 40\% & 50\% & 60\% & 70\% & 80\% \\
\hline
Grad. & 13.24 & 13.05 & 13.01 & 12.61 & 12.49 & 12.20 \\ 
Mixed & 13.58 & 13.38 & 13.27 & 13.03 & 12.86 & 12.50 \\ 
Biharm. & 13.36 & 13.10 & 12.23 & 12.88 & 12.89 & 12.56 \\ 
 \hline
\end{tabular}
\label{tab:6}
\end{center}
\end{table}
\subsection{Continuous functions}

We would now like to recover continuous functions. We define a tensor product partition $\mathcal{T}_h$ of the square $[-1,1] \times [-1,1]$ using the set of points 
\begin{align*}
\mathcal{N}_h= \left\lbrace (a_i,b_j) \right\rbrace_{i=1,j=1}^{n,m}, \quad
\text{where } a_i = \frac{2(i-1)}{n-1} - 1, \;
\text{and } b_j = \frac{2(j-1)}{m-1} - 1.
\end{align*}

We sample the function value at each point in the partition, and then apply Gaussian noise of variance $0.05$. We then refine the partition several times, which halves the mesh size $h$ in each iteration.
We will now consider the first test function. Let the function $f$ be given by
$f(x,y) = \sin(3x) e^{x^2 - y^2} $
over the domain $[-1,1] \times [-1,1]$ (see Figure \ref{fig:18}). 

We compare the PSNR values for the recovered function and the original function 
for different steps of refinement in Table \ref{tab:9}, where the refinement step is given by 
the step-size $h$.

\begin{table}[H]
\begin{center}
\caption{PSNR for $f$ using different penalty terms}
\begin{tabular}{ |c|c|c|c|c| }
\hline
$i$ & $h$ & Grad. & Mixed & Bihar. \\
\hline
0 & 2/19 & 19.85 & 24.31 & 25.54 \\ 
1 & 1/19 & 20.01 & 24.48 & 26.44 \\ 
2 & 1/38 & 19.59 & 24.61 & 26.61 \\ 
3 & 1/76 & 19.16 & 24.67 & 26.69 \\
4 & 1/152 & 18.51 & 24.70 & 26.72 \\
5 & 1/304 & 17.95 & 24.71 & 26.73 \\
 \hline
\end{tabular}
\label{tab:9}
\end{center}
\end{table}
We can see that PSNR values  do not increase or decrease for the spline with 
the mixed derivative penalty and the spline with the biharmonic penalty, whereas 
the PSNR values decrease for the spline with the gradient penalty. This is due to the fact that 
the stability constant depends on the mesh-size $h$ for the spline with 
the gradient penalty.


We show the functions recovered after the fifth iteration in Figure \ref{fig:18}. We can see that gradient penalty spline produces a recovered function that overfits the noisy data. On the other hand, both the mixed derivative and the biharmonic splines produce smoother recovered functions.

\begin{figure}[H]
\begin{center}
\begin{tabular}{ cc }
\includegraphics[scale=0.5,trim={1.5cm 0.9cm 1.2cm 1cm},clip]{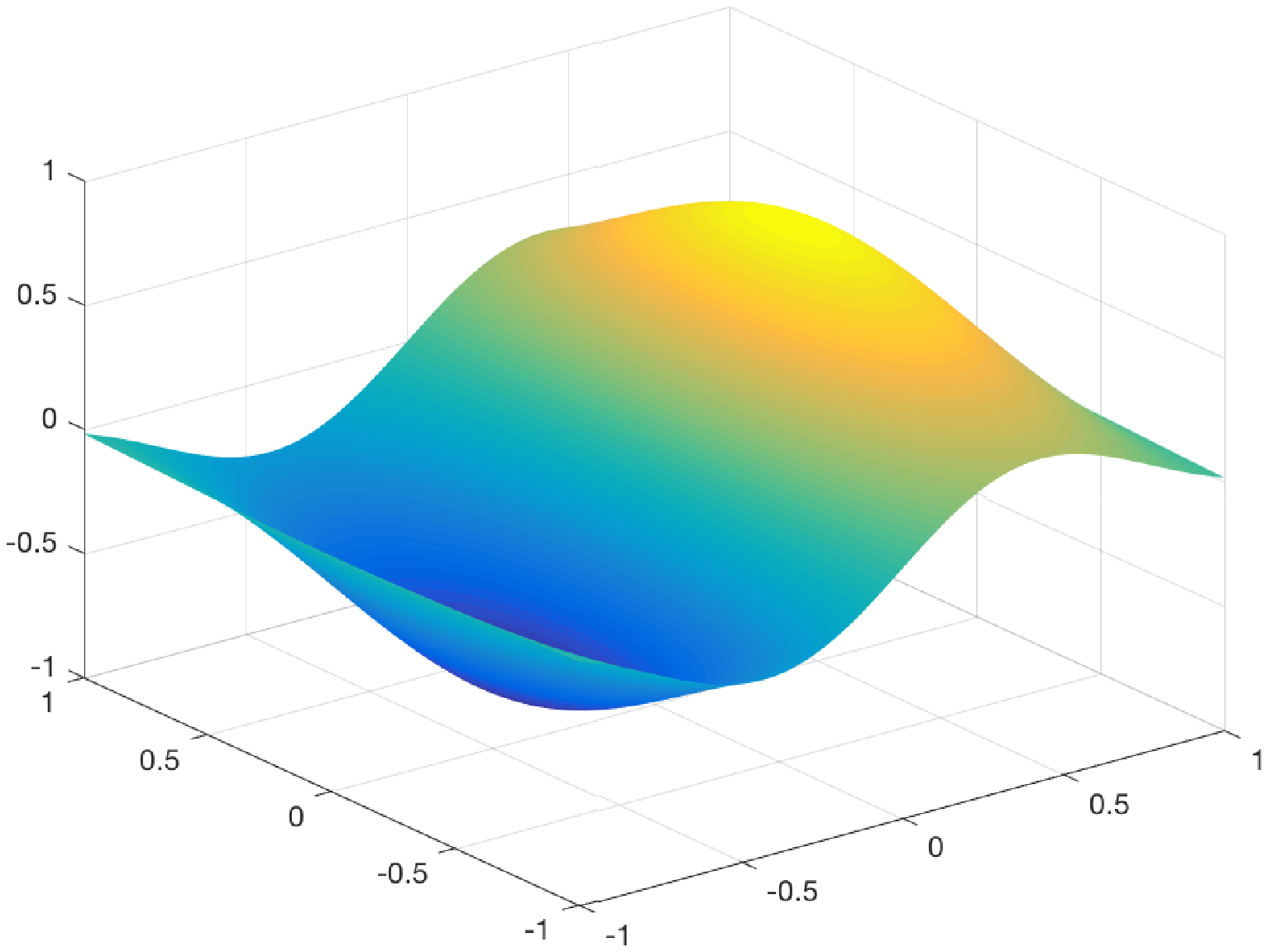}
&
\includegraphics[scale=0.5,trim={1.5cm 0.9cm 1.2cm 1cm},clip]{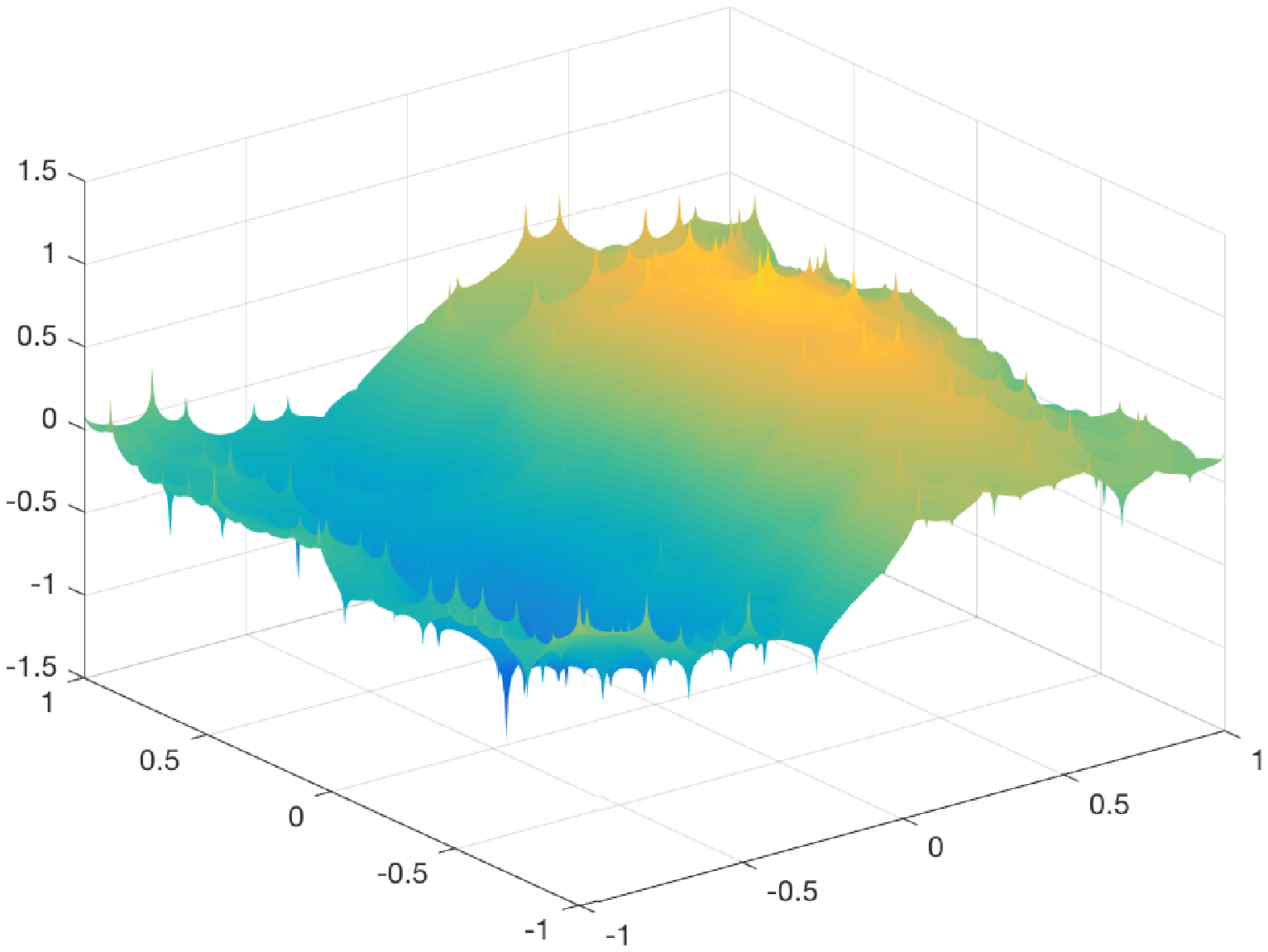}
\\\\
\includegraphics[scale=0.5,trim={1.5cm 0.9cm 1.2cm 1cm},clip]{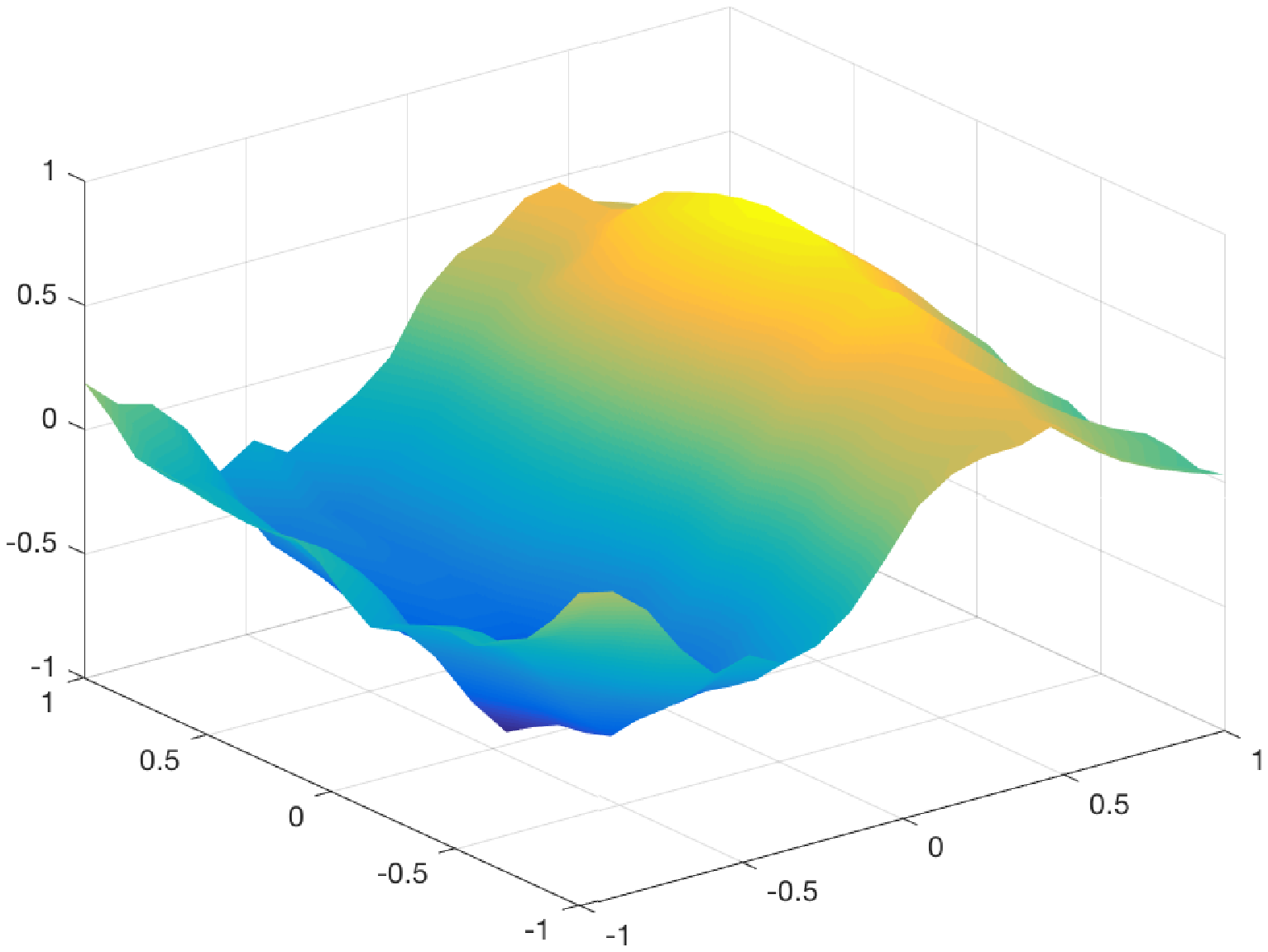}
&
\includegraphics[scale=0.5,trim={1.5cm 0.9cm 1.2cm 1cm},clip]{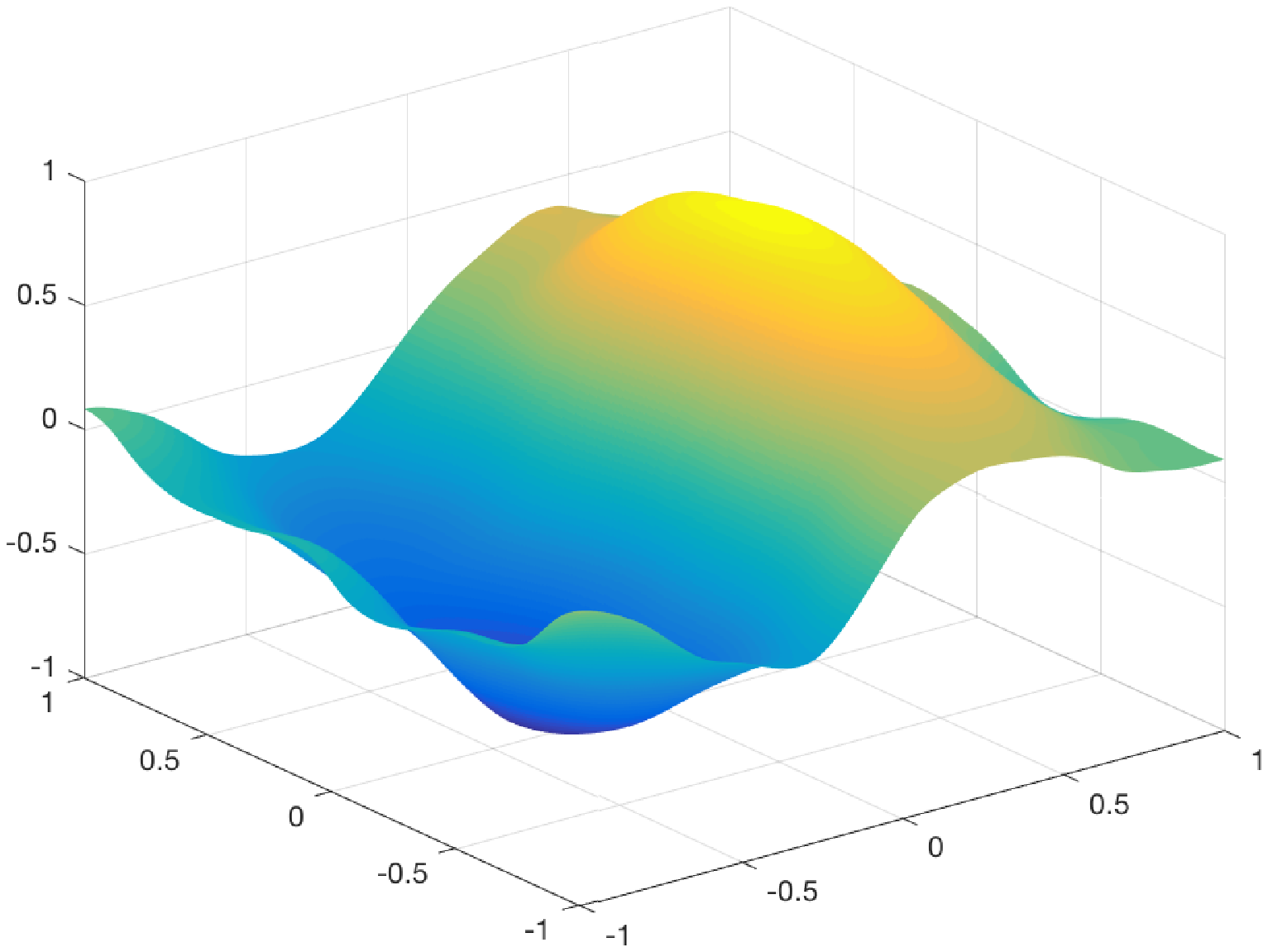}
\end{tabular}
\end{center}
\caption{$f(x,y) = \sin(3x) e^{x^2 - y^2}$ restricted to $[-1,1]\times[-1,1]$ (top left), function recovered using gradient penalty spline (top right), function recovered using mixed derivative spline (bottom left), function recovered using biharmonic spline (bottom right).}
\label{fig:18}
\end{figure}

We want to see the effect of the mesh-size on the spurious spikes of the recovered 
function. 
In Figure \ref{fig:f grad}, we show the functions recovered using the gradient penalty spline  using the coarser mesh-sizes 
$h=1/19$ and $h=1/76$. These pictures how that the spurious spikes are still present although 
the the spikes are slightly smoother in the coarser mesh results. 

\begin{figure}[H]
\begin{center}
\begin{tabular}{ cc }
\includegraphics[scale=0.5,trim={1.5cm 0.9cm 1.2cm 1cm},clip]{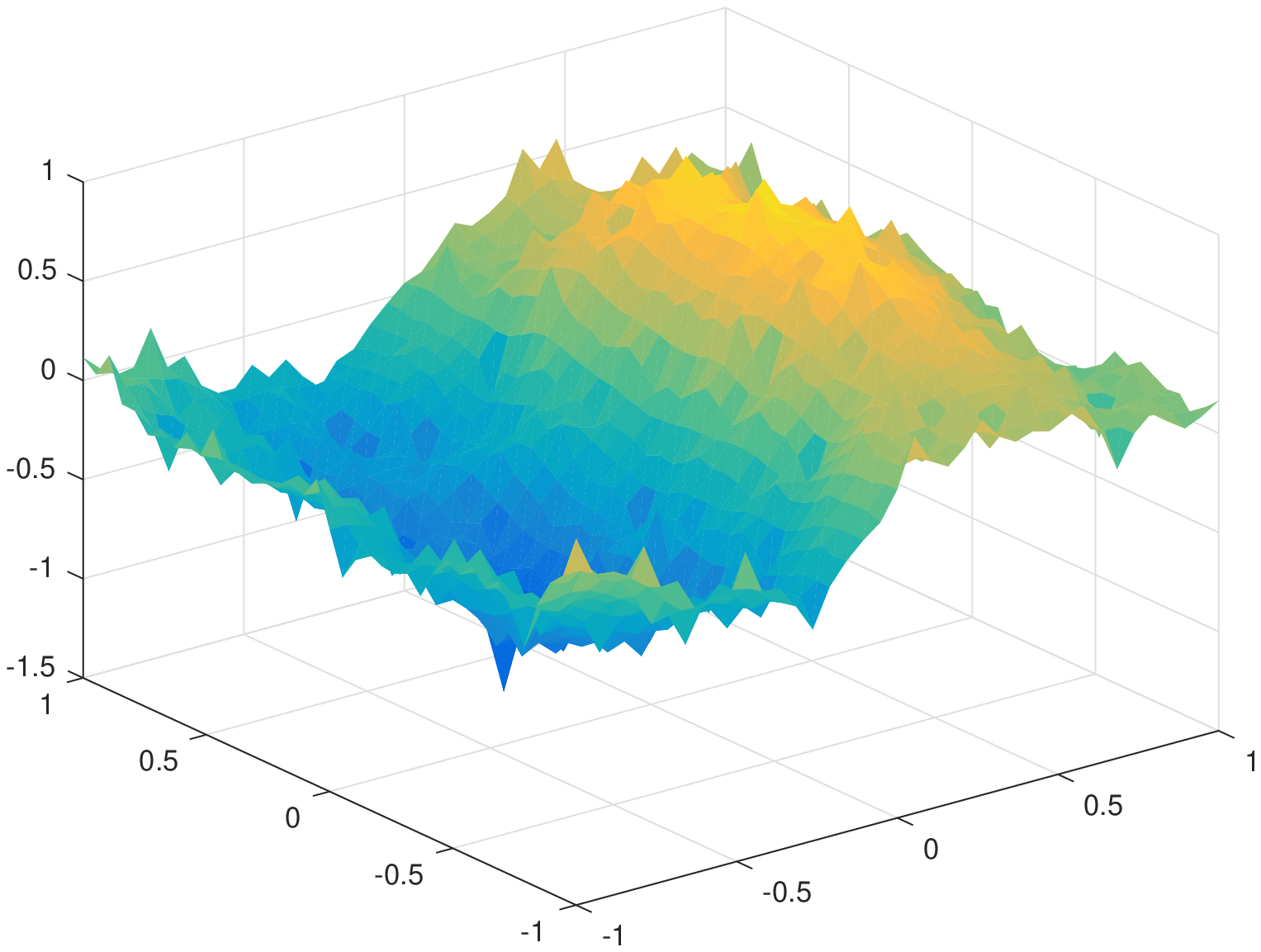}
&
\includegraphics[scale=0.5,trim={1.5cm 0.9cm 1.2cm 1cm},clip]{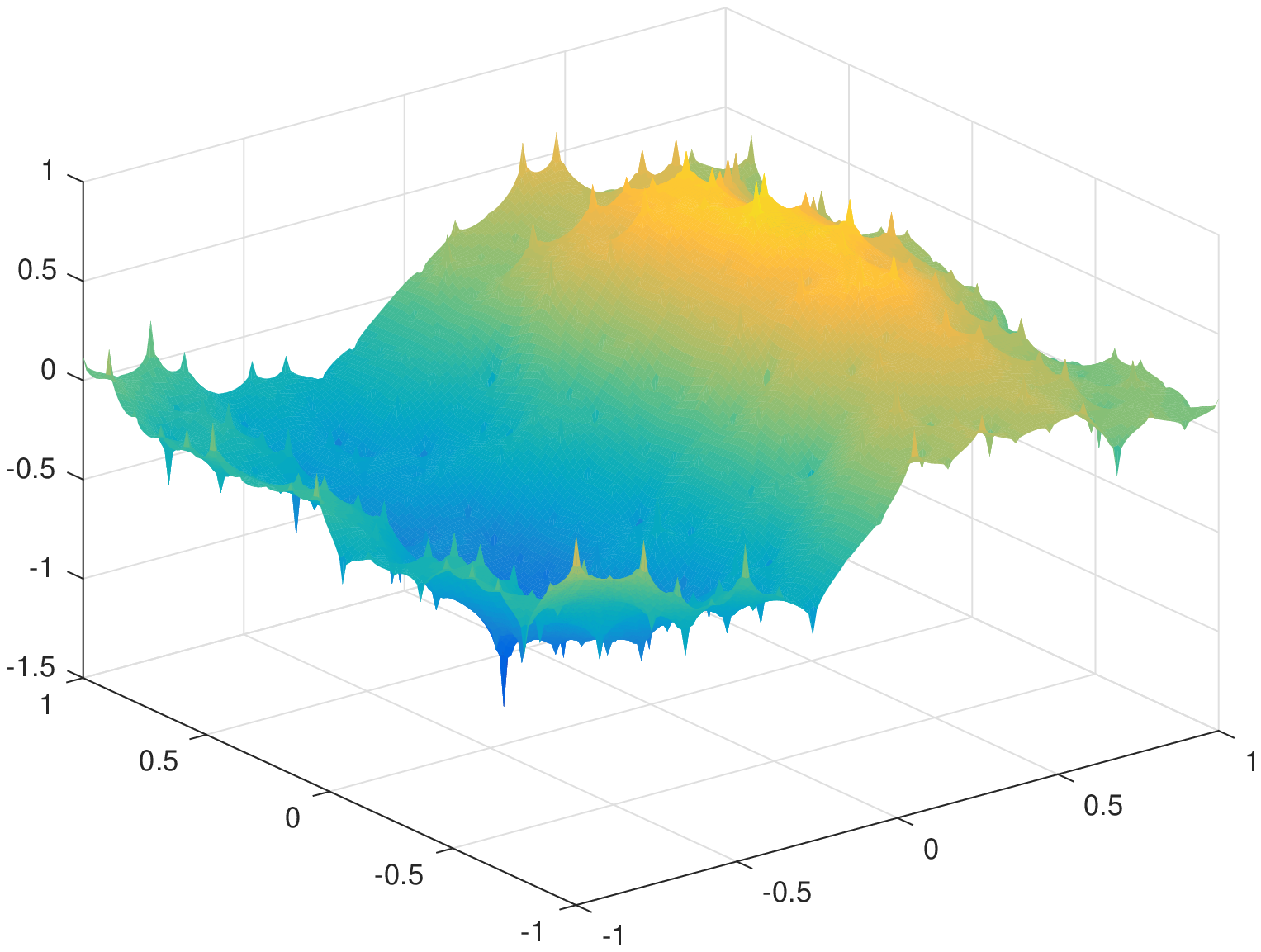}
\end{tabular}
\end{center}
\caption{Function recovered using gradient 
penalty spline with $h=1/19$ (left), function recovered using gradient penalty spline with $h=1/76$ (right)}
\label{fig:f grad}
\end{figure}

We will now provide a second test function. Let the function $g$ be given by
$g(x,y) = -x^2 -xy^2 $
over the domain $[-1,1]\times[-1,1]$ (see Figure \ref{fig:19}). 
We have tabulated the PSNR values for different splines at different levels of refinement 
in Table \ref{tab:10}. 
The results are similar to the first example but the spurious spikes of
the recovered function using the gradient penalty formulation have not affected the 
PSNR values much in this example.

\begin{table}[H]
\begin{center}
\caption{PSNR for $g$ using different penalty terms}
\begin{tabular}{ |c|c|c|c|c| }
\hline
$i$ & $h$ & Grad. & Mixed & Biharm. \\
\hline
0 & 2/19 & 19.49 & 21.24 & 22.19 \\ 
1 & 1/19 & 19.01 & 21.17 & 22.29 \\ 
2 & 1/38 & 18.76 & 21.15 & 22.33 \\ 
3 & 1/76 & 18.33 & 21.13 & 22.31 \\
4 & 1/152 & 18.16 & 21.12 & 22.30 \\
5 & 1/304 & 18.82 & 21.11 & 22.29 \\
 \hline
\end{tabular}
\label{tab:10}
\end{center}
\end{table}


We show the functions recovered after the fifth iteration in Figure \ref{fig:19}. Again, we see that the gradient penalty spline overfits the data.

\begin{figure}[H]
\begin{center}
\begin{tabular}{ cc }
\includegraphics[scale=0.5,trim={1.5cm 0.9cm 1.2cm 1cm},clip]{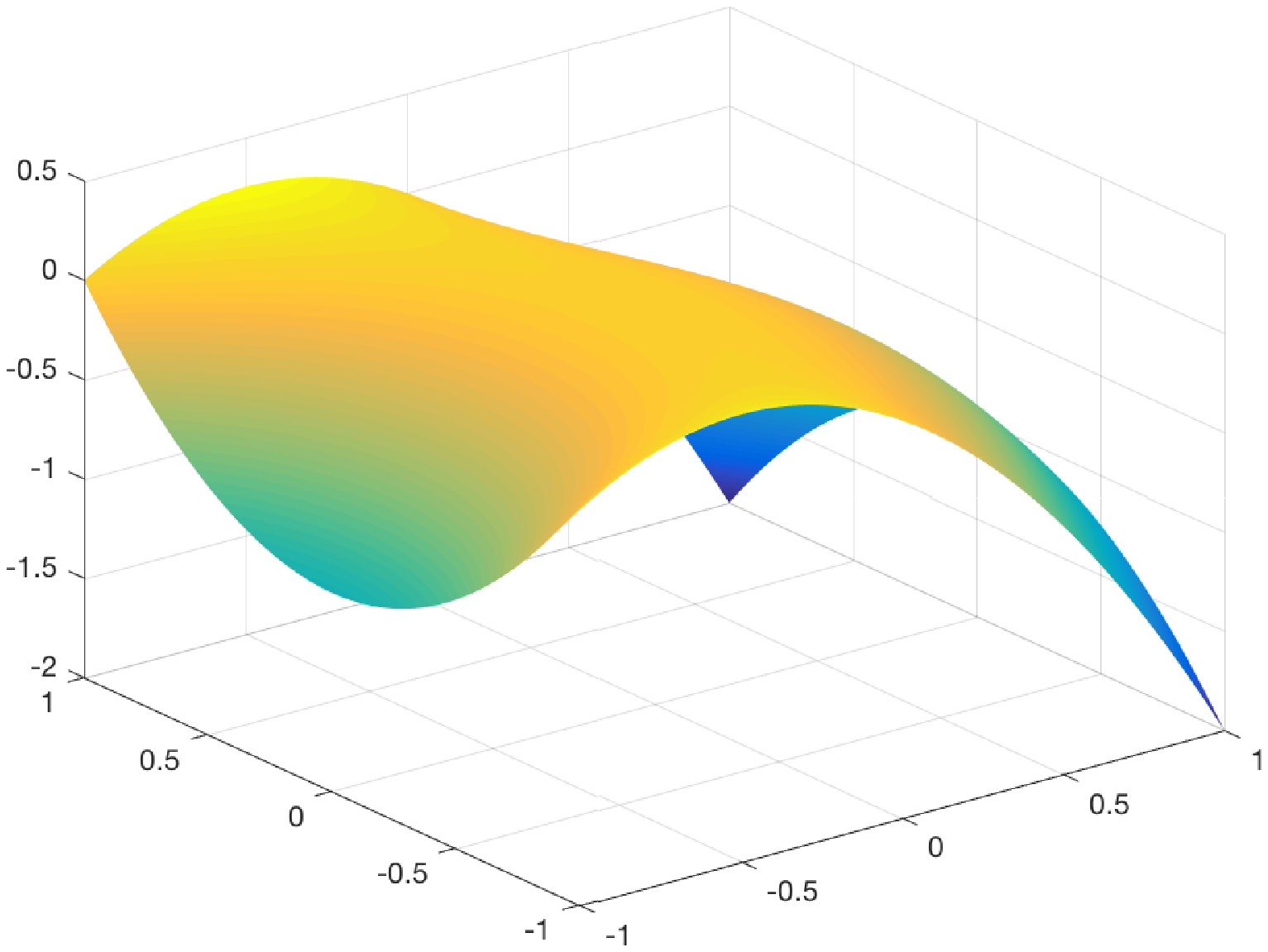}
&
\includegraphics[scale=0.5,trim={1.5cm 0.9cm 1.2cm 1cm},clip]{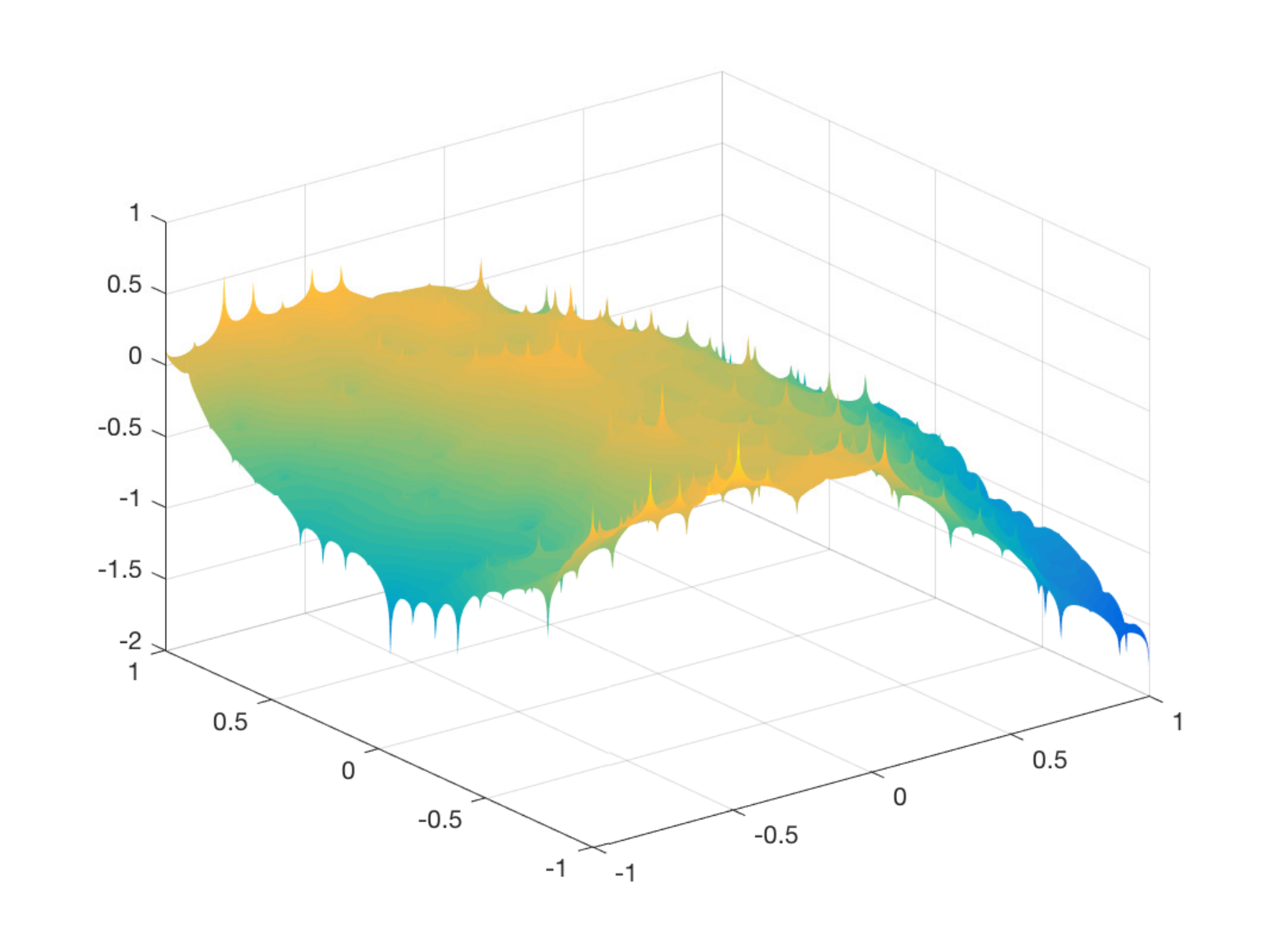}
\\\\
\includegraphics[scale=0.5,trim={1.5cm 0.9cm 1.2cm 1cm},clip]{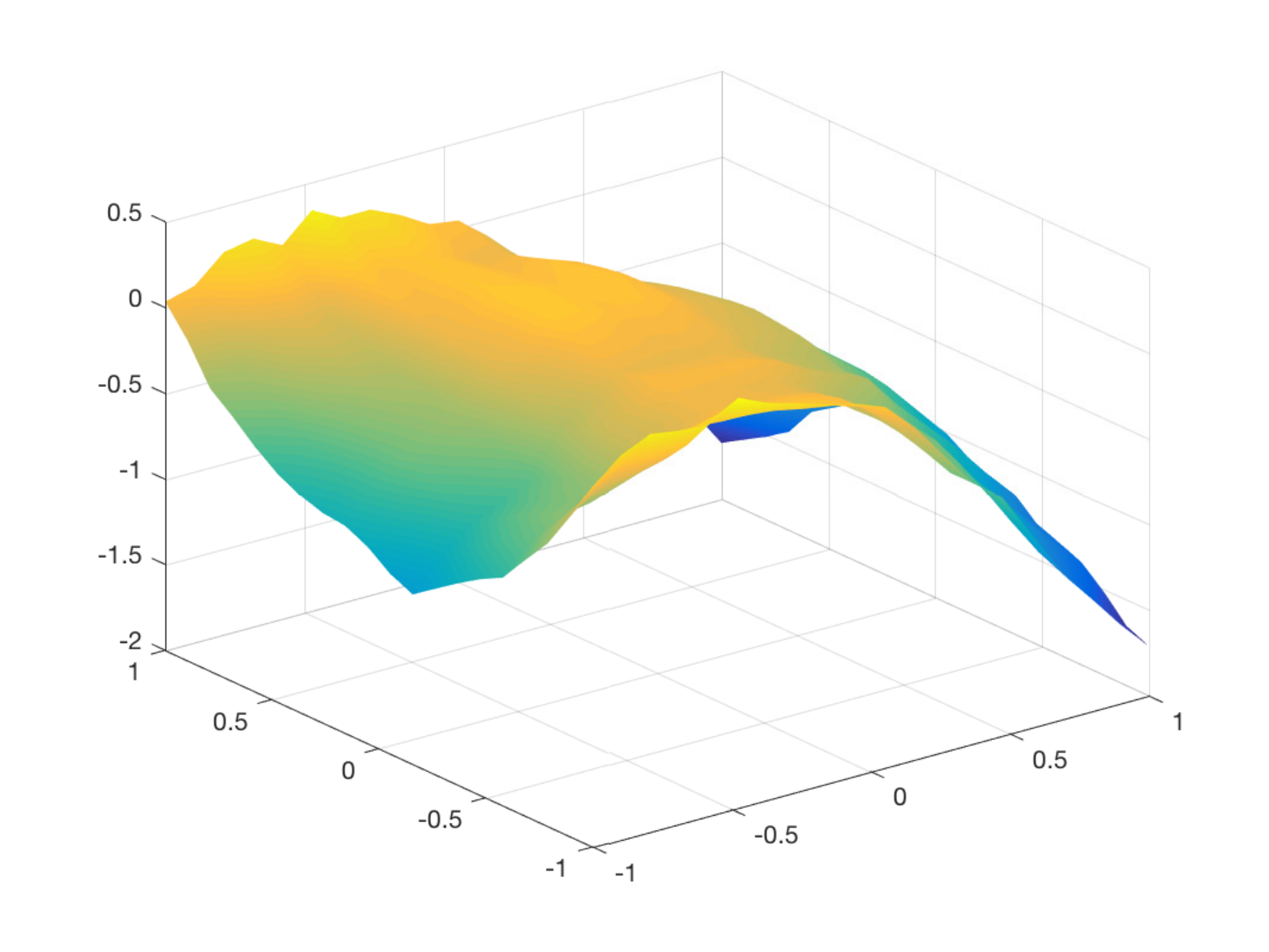}
&
\includegraphics[scale=0.5,trim={1.5cm 0.9cm 1.2cm 1cm},clip]{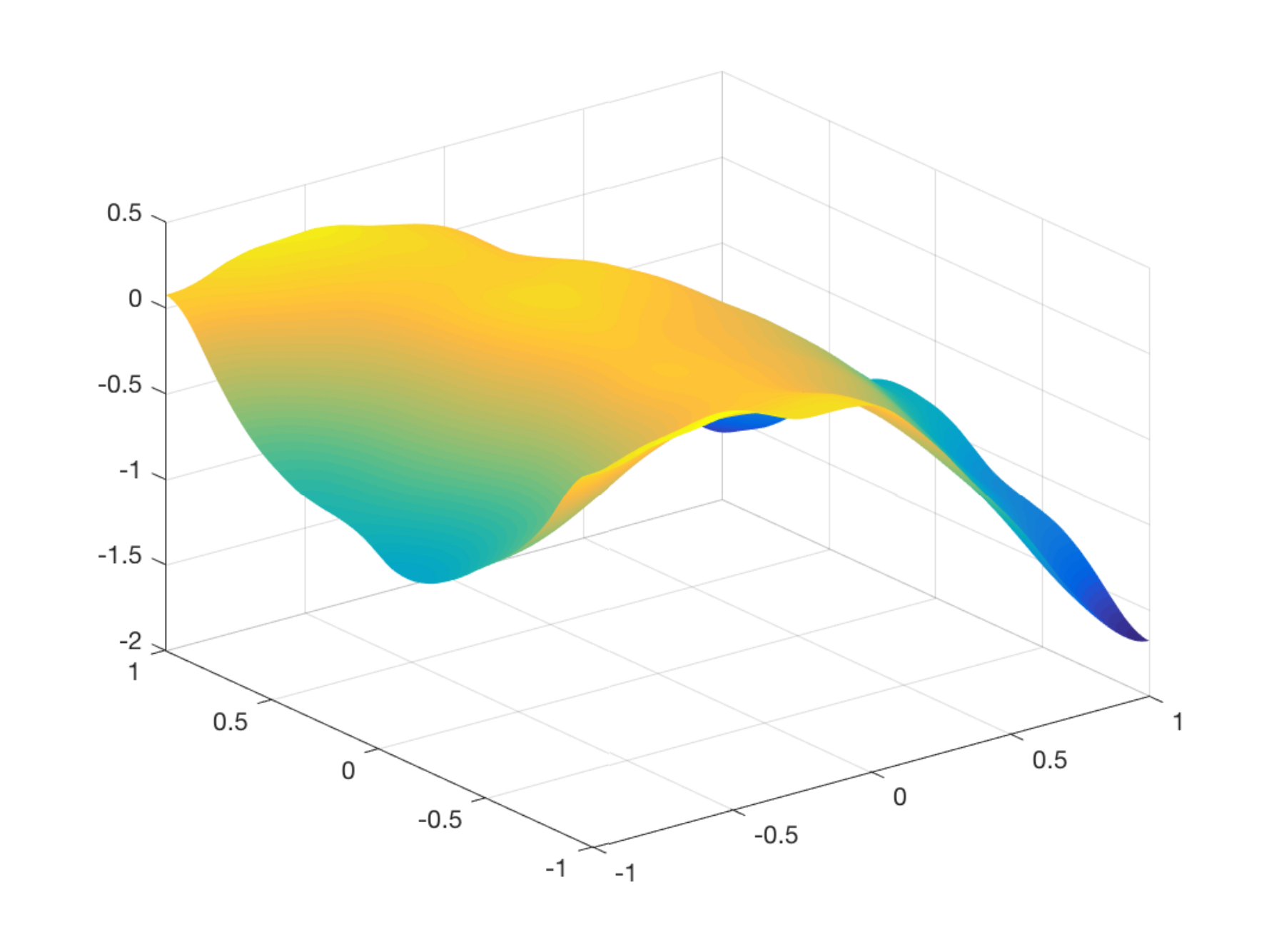}
\end{tabular}
\end{center}
\caption{$g(x,y) = -x^2 - xy^2$ restricted to $[-1,1]\times[-1,1]$ (top left), function recovered using gradient penalty spline  (top right), function recovered using mixed derivative spline (bottom left), function recovered using biharmonic spline  (bottom right).}
\label{fig:19}
\end{figure}
\section{Discussion}

We compared three different bivariate L-spline approaches for removing the mixture of Gaussian and impulsive noise from images. We found that for the Lena image, the biharmonic penalty produced recovered images with the best PSNR. However, we found that the biharmonic penalty performed the worst when recovering the Baboon image. The gradient and mixed derivative penalties performed very similarly to each other when recovering the two real life images. For the Binary image, we found that the mixed derivative penalty performed the best, followed by the biharmonic penalty and then the gradient penalty.


We then applied the same approaches to recover two continuous functions from a set of noisy observations. We found that for both functions, the biharmonic penalty produced the best recovered functions, closely followed by the mixed derivative penalty. The gradient penalty produced recovered functions that overfitted the data.

The overfitting occurred because the gradient penalty formulation is not well-posed in the continuous setting. For dimensions $d \geq 2$, we have that $H^1(\Omega) \not\subset C^0(\Omega)$ (by the Sobolev embedding theorem). This ill-posedness exhibits itself when the mesh size goes to zero \cite{Utr88}. The other formulations, however, are well-posed in the continuous setting. This is because $H_m^1(\Omega) \subset C^0(\Omega)$ for any dimension $d$ \cite{ST87}, and $H^2(\Omega) \subset C^0(\Omega)$ for dimensions $d \leq 3$. 

Overall, the gradient penalty was the simplest spline to implement and the most computationally efficient. However, 
as this is not well-posed for $d>1$, it often produces spurious results. 
The computational cost of  the spline with the mixed derivative penalty is very close to the gradient penalty
and this is well-posed for all dimensions \cite{Lam14}. 
The biharmonic penalty was the least simple to implement, and was the least computationally efficient. 
Moreover, the spline with the biharmonic penalty is also not well-posed when $d>3$. 
Therefore, we find that the spline with the mixed derivative penalty is the best choice among the presented three splines.

\bibliography{signal}
\bibliographystyle{plain}
\end{document}